\pdfoutput=1

\documentclass{article}
\usepackage[utf8]{inputenc}
\usepackage{amsthm}
\usepackage{amsmath}
\usepackage{amsfonts}
\usepackage{amssymb}
\usepackage{tikz-cd}
\usepackage{biblatex}
\usepackage{cleveref}
\usepackage{color}

\newtheorem{lemma}{Lemma}[subsection]
\Crefname{lemma}{Lemma}{Lemmas}
\newtheorem{prop}[lemma]{Proposition}
\Crefname{prop}{Proposition}{Propositions}
\newtheorem{thm}[lemma]{Theorem}
\Crefname{thm}{Theorem}{Theorems}
\newtheorem*{thm*}{Theorem}
\newtheorem{coro}[lemma]{Corollary}
\Crefname{coro}{Corollary}{Corollary}

\theoremstyle{definition}
\newtheorem{dfn}[lemma]{Definition}
\Crefname{dfn}{Definition}{Definitions}
\newtheorem*{dfn*}{Definition}
\newtheorem{rk}[lemma]{Remark}
\Crefname{rk}{Remark}{Remarks}

\Crefname{fact}{Fact}{Fact}
\newtheorem*{fact*}{Fact}

\Crefname{eg}{Example}{Examples}
\newtheorem*{q}{Question}
\Crefname{q}{Question}{Question}

\newcommand{\Z}{\mathbb{Z}}
\newcommand{\R}{\mathbb{R}}

\newcommand{\mi}{\setminus}

\newcommand{\1}{^{-1}}

\title{Scattering manifolds and symplectic fillings}
\author{Davide Alboresi}
\date{}

\begin{document}

\maketitle

\begin{abstract}
Scattering symplectic manifolds are (closed) manifolds with a mildly degenerate Poisson structure. In particular they can be viewed as symplectic structures on a Lie algebroid which is almost everywhere isomorphic to the tangent bundle. In this paper we prove that all scattering symplectic manifolds arise as glueings of weak symplectic fillings of contact manifolds, and all pairs of weak symplectic fillings with matching boundaries can be glued to a scattering symplectic manifold.
\end{abstract}

\tableofcontents

\section{Introduction}

A scattering symplectic structure on a manifold $X$ is a Poisson bivector $\pi$ which is almost everywhere symplectic, and whose rank drops on a codimension-$1$ submanifold $Z$, in a prescribed way. The prescription can be described in terms of the \textit{scattering Lie algebroid}, which is a vector bundle $^\text{sc}TX$ such that $\Gamma(^\text{sc}TX)=\langle x^2\partial_x,\,x\partial_{y_1},\,\dots,\,x\partial_{y_k} \rangle$, where $y_1,\dots,y_k$ is a coordinate system on $Z$, and locally $Z=\{x=0\}$. The bundle $^\text{sc}TX$ comes with a map $\rho: {^\text{sc}TX}\longrightarrow TX$, induced from the inclusion on sections, which is an isomorphism on $X\mi Z$. We say that $\pi$ is scattering symplectic if there exists a symplectic form $\omega:{^\text{sc}TX}\longrightarrow {^\text{sc}T^*X}$ such that the diagram
\begin{equation}\label{diagram}
\begin{tikzcd}
 ^\text{sc}T^*X \arrow{r}{\omega^{-1}} \arrow{d}[swap]{\rho} & ^\text{sc}TX \arrow{d}{\rho}\\
 T^*X \arrow{r}{\pi} & TX 
\end{tikzcd}
\end{equation}
commutes. Scattering symplectic structures were introduced in \cite{melinda}, the idea being to view $\omega$ (or rather its inverse $\omega\1$) as a desingularization of $\pi$, and that much of the Poisson geometry of $\pi$ is actually encoded in the symplectic geometry of $\omega$. Notably, one is able to prove normal form theorems using Moser-trick type of arguments, and to compute Poisson cohomology. Scattering symplectic structures are examples of Poisson structures that can be ``desingularized by a Lie algebroid", i.e. for which there exists a Lie algebroid with a symplectic structure that can be fit in a diagram like (\ref{diagram}) (see also \cite{ralphthesis}). Log-symplectic structures (\cite{guilleminmirandapires}) and elliptic symplectic structures (\cite{gilmarco}) are other examples of this.\\
A scattering symplectic structure on $(X,\,Z)$ induces a canonical cooriented contact structure on $Z$. As in the log-symplectic case, then, $X$ is orientable if and only if $Z$ is coorientable in $X$. Moreover, there is a normal form theorem around the singularity: every scattering symplectic structure looks like
\begin{equation}\label{normalform}
    \omega=\frac{d\lambda}{\lambda^3}\wedge (-\alpha+x^2\beta_1)+\frac{1}{2x^2}d\alpha+\beta_2
\end{equation}
Here $\lambda$ is a distance function from the zero section of the normal bundle $NZ$, $\alpha$ is a contact form on $Z$, and $\beta_1$ and $\beta_2$ are closed differential forms on $Z$. The cohomology classes of the forms $\beta_i$ are invariants of the Poisson structure.\\ 
When $\beta_1=\beta_2=0$, and $X$ is oriented, this is the normal form of a strong symplectic filling of $(Z,\,\ker\alpha)$; conversely given a filling $W$ of $(Z,\,\ker\alpha)$ one can glue two copies of $W$ with opposite orientations and get a scattering symplectic manifold with singularity on $Z$, and $\beta_1=\beta_2=0$. This can be phrased as follows:
\begin{fact*}[\cite{melinda}]
$(Z,\,\xi=\ker\alpha)$ is the singular locus of an oriented scattering symplectic manifold with $\beta_1=\beta_2=0$ if and only if it is strongly fillable.
\end{fact*}
Motivated by the above result, we give the following definition:
\begin{dfn*}\label{scfill}
A cooriented contact manifold $(Z,\,\xi)$ has an \textit{(orientable) scattering-filling} if it is the singular locus of a closed (orientable) scattering symplectic manifold.
\end{dfn*}
The question that we want to address here is:
\begin{q}
How does scattering-fillability compare to the other notions of fillability existing in the contact topology literature?
\end{q}
The main result of this paper is the following.
\begin{thm*}
The cooriented contact manifold $(Z,\,\ker\alpha)$ has an orientable scattering-filling if and only if it is weakly fillable, in the sense of \cite{weakfillings}.
\end{thm*}
Also non-orientable scattering manifolds can be described in terms of symplectic fillings.
\begin{thm*}
Assume that $(Z,\,\ker\alpha)$ is the singular locus of a non-orientable scattering manifold. Then there exists a connected double cover $p:\tilde{Z}\longrightarrow Z$ such that $(\tilde{Z},\,p^*\alpha)$ is weakly fillable, in the sense of \cite{weakfillings}.
\end{thm*}
\paragraph{Organization of the paper.}
In section \ref{1} we introduce scattering symplectic manifolds and their local forms. In section \ref{2} we recall some notions of symplectic fillings of contact manifolds, and in section \ref{3} we prove the main theorems. In section \ref{4} we deduce some simple results on the symplectic topology of scattering symplectic manifolds. 

\paragraph{Acknowledgements.}
The author is thankful to Melinda Lanius for useful conversations. This research was supported by the VIDI grant 639.032.221 from NWO, the Netherlands Organisation for Scientific Research.

\section{Scattering symplectic manifolds}\label{1}

\subsection{The scattering tangent bundle}

Let $(X,\,Z)$ be a pair consisting of a manifold $X$ with a codimension-$1$ closed submanifold $Z$. Let $^{\text{sc}}\mathfrak{X}_Z$ be the sheaf of vector fields vanishing on $Z$, quadratically in the transverse direction, linearly in the tangent direction. Explicitely, for each point $p\in Z$ choose a coordinate neighbourhood $U$, with coordinates $(x,\,y_1,\,\dots,y_N)$, such that $Z\cap U=\{x=0\}$, and $(y_1,\,\dots,y_N)$ is a coordinate system for $Z$ on $Z\cap U$. We define $^{\text{sc}}\mathfrak{X}_Z(U):=\langle x^2\partial_x,\,x\partial_{y_1},\,\dots,\,x\partial_{y_N}\rangle$. This defines a locally free sheaf $^{\text{sc}}\mathfrak{X}_Z$ of rank $\text{dim}X$.
\begin{dfn}\label{defscattering}
The \textit{scattering tangent bundle} $^{\text{sc}}TX={^{\text{sc}}TX}_Z$ is the unique vector bundle such that $\Gamma(^{\text{sc}}TX)=$ $^{\text{sc}}\mathfrak{X}_Z$. Its dual is denoted with $^{\text{sc}}T^*X$ and called the \textit{scattering cotangent bundle}
\end{dfn}
Since $^{\text{sc}}\mathfrak{X}_Z$ is closed under the Lie bracket, $^{\text{sc}}TX$ is a Lie algebroid, with anchor map induced by the inclusion $^{\text{sc}}\mathfrak{X}\hookrightarrow \mathfrak{X}$. Moreover, the anchor induces an isomorphism between ${^{\text{sc}}TX_Z}|_{X\mi Z}$ and $TX|_{X\mi Z}$.

\subsection{Scattering symplectic structures}

We collect here the definition and some known results on scattering symplectic manifolds. All the results mentioned here are due to \cite{melinda}.

\begin{dfn}
A \textit{scattering symplectic structure} is a symplectic structure on $^{\text{sc}}TX$. That is, it is a closed non-degenerate section of $\Lambda^2(^{\text{sc}}T^*X)$. Here ``closed" means closed in the Lie algebroid sense, or equivalently as an ordinary differential form on $X\setminus Z$.
\end{dfn}
We will denote a scattering symplectic manifold with a triple $(X,\,Z,\,\omega)$. A scattering symplectic structure induces an ordinary symplectic structure on $X\mi Z$. In particular $\text{dim}X=2n$, and $X\mi Z$ is oriented. 
\begin{rk}
In \cite{melinda}, the attention is restricted to pairs of oriented manifolds $(X,\,Z)$. However, none of the results that we will mention depends on the orientability of $X$.
\end{rk}
The following propositions are proven in \cite{melinda}, in the orientable case. The same proofs given in \cite{melinda} hold in the non-orientable case.
\begin{prop}
A scattering symplectic structure $\omega$ on $(X,\,Z)$ induces a canonical cooriented contact structure $\xi$ on $Z$. In particular, $Z$ is oriented.
\end{prop}
A contact form can be recovered as follows: pick a fiber metric on the normal bundle, and denote with $\lambda$ the distance from $Z$ function. Write 
\[
\omega=\alpha_\lambda\wedge\frac{d\lambda}{\lambda^3}+\frac{\beta_\lambda}{\lambda^2}
\]
It follows from closure and nondegeneracy of $\omega$ that $\alpha={\alpha_\lambda}|_{TZ}$ is a contact form. The form $\alpha$ does not depend on the choice of the metric up to a positive factor. Hence there is a well-defined induced cooriented contact structure. Moreover, a scattering symplectomorphism preserves the contact form up to a positive scaling factor, hence scattering symplectomorphisms preserve the cooriented contact structure on the singular locus (i.e. induce orientation preserving contactomorphisms on the singular locus).

\subsection{A normal form theorem around the singular locus}

In order to state the normal form theorem around the singular locus, we need to spend a word on tubular neighbourhoods of hypersurfaces. Identify a neighbourhood of $Z$ with an open neighbourhood of the zero section in the normal bundle $NZ$. Let $\lambda$ be a distance function from $Z$, induced by a fiber metric on $NZ$. Define $\tilde{Z}:=\{\lambda=1 \}$. Fiberwise multiplication by $-1$ induces a diffeomorphism $\sigma:\tilde{Z}\longrightarrow \tilde{Z}$. This induces a free $\Z_2$ action, and $\tilde{Z}/\Z_2=Z$. Moreover, one has $NZ\cong \R\times_{\Z_2}\tilde{Z}$, where $\R\times_{\Z_2}\tilde{Z}:=(\R\times \tilde{Z})/\Z_2$, with action $(-1)\cdot (t,\,x)=(-t,\,\sigma(x))$.\\
When $NZ$ is trivial, then $\tilde{Z}=Z\sqcup Z$, and one gets back $NZ\cong \R\times Z$. Moreover, one can always assume that a tubular neighbourhood with smooth boundary is of the form $(-\varepsilon,\,\varepsilon)\times_{Z_2}\tilde{Z}$.

\begin{thm}\label{generalnormalform}
Let $(X,\,Z,\,\omega)$ be scattering symplectic. Take a fiber metric on the normal bundle of $Z$, denote the distance function from $0$ with $\lambda$. Let $\tilde{Z}=\{\lambda=1\}$, and denote the projection with $p:\tilde{Z}\longrightarrow Z$. There exist
\begin{itemize}
\item a contact form $\alpha$ on $Z$
\item a closed $1$-form $\beta_1$ on $Z$ 
\item a closed $2$-form $\beta_2$ on $Z$
\end{itemize}
such that, in a tubular neighbourhood of the form $(-\varepsilon,\,\varepsilon)\times_{Z_2}\tilde{Z}$,
\begin{equation}\label{equationgeneralnormalform}
\omega\cong \frac{d\lambda}{\lambda^3}\wedge (-p^*\alpha+\lambda^2p^*\beta_1)+\frac{1}{2\lambda^2}dp^*\alpha+p^*\beta_2
\end{equation}
Here the expression $\omega\cong\eta$ means: there exists a diffeomorphism $\phi$ such that $\phi|_Z=id$, and such that $\phi^*\omega=\eta$ on a tubular neighbourhood of $Z$.
\end{thm}

\section{Symplectic fillings of contact manifolds}\label{2}

In this section we recall several different notions of symplectic filling in contact geometry. In general, a symplectic filling of a cooriented contact manifold $(Z,\,\xi)$ is a connected symplectic manifold $(W,\,\omega)$ with boundary $\partial W=Z$ as oriented manifolds (i.e., the boundary orientation matches the contact orientation) such that the symplectic form is ``compatible" with the contact distribution.\\
Different ways of making the word ``compatible" precise give different flavours of symplectic fillings. We will always consider notions of compatibility that only involve a neighbourhood of the boundary. Let us start with the strongest notion appearing in this paper.
\begin{dfn}
$(W,\,\omega)$ is a \textit{strong symplectic filling} of $(Z,\,\xi)$ if there is a collar neighbourhood $U=(-\varepsilon,\,0]\times Z$ of $\partial W$, with $\partial W=\{0\}\times Z$, and transverse coordinate $s$, such that $\omega=d(s\alpha)$ on $U$.
\end{dfn}
In particular $\omega$ pulls-back to $\xi$ on $Z$ as a symplectic form, belonging to the canonical conformal symplectic structure of $\xi$.\\
Denote the conformal symplectic structure on $\xi$, viewed as a set of symplectic forms, with $CS_\xi$. 
\begin{dfn}[\cite{weakfillings}]
$(W,\,\omega)$ is a \textit{weak symplectic filling} of $(Z,\,\xi)$ if $\omega|_\xi$ is symplectic, and for all representatives $\eta$ of $CS_\xi$, the form $\omega|_{\xi}+\eta$ is symplectic on $\xi$.
\end{dfn}
Of course strong implies weak. In dimension $4$ the contact distribution is $2$-dimensional, and this definition of weak filling is equivalent to the requirement that $\omega|_{\xi}$ is a positive form, i.e. it induces the same orientation as the contact orientation of $\xi$. This is equivalent, still in dimension $4$, to $\omega|_{\xi}$ belonging to $CS_\xi$. In higher dimension, only requiring $\omega|_{\xi}\in CS_\xi$ turns out to be equivalent to the existence of a strong filling (\cite{McDuff1991}). Instead, in all dimensions (including $4$) the notions of strong and weak fillability are different (\cite{weakfillings}). The notion of weak fillability can be understood in terms of almost complex structures.
\begin{thm}[\cite{weakfillings}]\label{almostcomplex}
A symplectic manifold $(W,\,\omega)$ with boundary $\partial W= Z$ is a weak filling if and only if there exists an almost complex structure $J$ on $W$ such that
\begin{itemize}
\item $J$ tames $\omega$
\item $\xi=TZ\cap JTZ$
\item for any contact form $\alpha$ and $v\in \xi$, $d\alpha(v,\,Jv)>0$ (with respect to the boundary orientation)
\end{itemize}
\end{thm}
A possibly ``intermediate`` notion, motivated by dynamics and the theory of holomorphic curves is that of a \textit{stable filling}, where the boundary is required to inherit a \textit{stable Hamiltonian structure}, in a strong sense (meaning, in a way that induces a normal form in a collar neighbourhod of the boundary). Let us recall first what a stable Hamiltonian structure is.
\begin{dfn}
A \textit{stable Hamiltonian structure} on a $(2n-1)$-dimensional manifold $Z$ is a pair consisting of a one-form and a two-form $(a,\,b)$ such that
\begin{itemize}
\item $db=0$
\item $a\wedge b^{n-1}\neq 0$
\item $\ker b\subset\ker da$
\end{itemize}
\end{dfn}
A consequence of the definition is that, for $\varepsilon$ small enough, the form $\Omega=d(ta)+b$ is symplectic on $(-\varepsilon,\,\varepsilon)\times Z$. This is the setup in which one is able to study moduli spaces of (punctured) holomorphic curves, in the framework of symplectic field theory (\cite{introsft}). A stable filling is what one expects.
\begin{dfn}
$(W,\,\omega)$ is a \textit{stable (Hamiltonian) filling} if there exists an outward pointing vector field $X$ aroud the boundary such that $(\iota_X\omega)|_{T\partial W},\,\omega|_{T\partial W})$ is a stable Hamiltonian structure.
\end{dfn}

\section{Scattering fillings versus weak fillings}\label{3}

\subsection{Statement of the result}

Scattering symplectic geometry suggests yet another notion of filling of a contact manifold.
\begin{dfn}
A \textit{(orientable) scattering filling} of a cooriented contact manifold $(Z,\,\ker\alpha)$ is a (orientable) scattering symplectic manifold $(X,\,Z,\,\omega)$, such that $Z$ with the induced contact structure is contactomorphic to $(Z,\,\ker\alpha)$. A scattering filling is \textit{strong} if the forms $\beta_1$, $\beta_2$ are exact.
\end{dfn}
In \cite{melinda} it is shown that 
\[
\text{strong orientable scattering fillability }\Leftrightarrow\text{ strong fillability}
\]
In this section we prove the main results of the present paper.
\begin{thm}\label{mainthm}
\begin{itemize}
\item $(Z,\,\ker\alpha)$ has an orientable scattering filling $\Leftrightarrow$ $(Z,\,\ker\alpha)$ is weakly fillable.
\item $(Z,\,\ker\alpha)$ has a non-orientable scattering filling $\Leftrightarrow$ there is a connected double cover $p:\tilde{Z}\longrightarrow Z$ such that $(\tilde{Z},\,p^*\ker\alpha)$ is weakly fillable.
\end{itemize}
\end{thm}

\subsection{Manifolds with scattering boundary}

When $X$ is a manifold with boundary, one can define the scattering tangent bundle $^\text{sc}TX$ exactly as in \Cref{defscattering}, with $Z=\partial X$. A scattering symplectic form on $X$ is defined accordingly as a symplectic form on $^\text{sc}TX$. When we mention \textit{symplectic manifolds with scattering boundary} we will always mean we are in the situation just described.\\
On a closed orientable scattering symplectic manifold $(X,\,Z,\,\omega)$, one can fix an orientation, and compare it with the symplectic orientation induced by $\omega$ on $X\mi Z$. These orientations are going to coincide on some components of $X\mi Z$, and differ on other components. This induces a splitting $X=X^+\cup X^-$, where $X^\pm$ are compact submanifolds with boundary $\partial X^\pm=Z$. Both $X^+$ and $X^-$ are symplectic manifolds with scattering boundary.\\
Conversely, any two symplectic manifolds with (connected) scattering boundary, and matching normal form around the boundary, can be glued along the boundary to form a (closed) scattering symplectic manifold. Moreover, given a symplectic manifold with scattering boundary $W$, it's easy to show that its double inherits a scattering symplectic structure with singular locus $\partial W$.\\
Hence, as far as the contact manifold is concerned, being the boundary of a symplectic manifold with scattering boundary or being the singular locus of a scattering symplectic manifold are equivalent notions. Let us formulate the result just explained in the following proposition.
\begin{prop}\label{boundaryscattering}
Let $(Z,\,\ker\alpha)$ be a cooriented contact manifold. Then it is oriented scattering fillable if and only if it can be realized as the boundary of a symplectic manifold with scattering boundary.
\end{prop}

\subsection{Reduction to the orientable case}

In this section we relate scattering fillability by a non-orientable manifold to orientable fillability of a double cover. The result is the following.
\begin{prop}\label{orientableisenough}
Let $(Z,\,\ker\alpha)$ be the singular locus of a non-orientable scattering symplectic manifold. Then there exists a double cover $p:\tilde{Z}\longrightarrow Z$, and symplectic manifold with scattering boundary that coincides with $(\tilde{Z},\,\ker p^*\alpha)$.
\end{prop}
\begin{proof}
This proposition is proven by constructing the real oriented blow-up along the singular locus $Z$.
\begin{lemma}
Given a non-orientable scattering symplectic manifold $(X,\,Z,\,\omega)$, there exists a symplectic manifold with scattering boundary $(\tilde{X},\,\tilde{Z},\,\tilde{\omega})$, with a map $b:(\tilde{X},\,\tilde{Z})\longrightarrow (X,\,Z)$ which restricts to a symplectomorphism $(\tilde{X}\mi \tilde{Z},\,\tilde{\omega})\cong (X\mi Z,\,\omega)$ on the regular part, and to a double cover on the singular part. This manifold coincides with the real oriented blow-up of $(X,\,Z)$ along $Z$.
\end{lemma}
\begin{proof}[Proof of the lemma]
It is immediate to see that $((-\varepsilon,\,\varepsilon)\times_{\Z_2}\tilde{Z})\mi (\{0\}\times_{\Z_2}\tilde{Z})=((-\varepsilon,\,\varepsilon)\times_{\Z_2}\tilde{Z})\mi (\{0\}\times Z)\cong (-\varepsilon,\,0)\times \tilde{Z}$. Moreover there is a smooth map $b:(-\varepsilon,\,0]\times \tilde{Z}\longrightarrow (-\varepsilon,\,\varepsilon)\times_{\Z_2}\tilde{Z}$, which on $\{0\}\times \tilde{Z}$ is the projection $p:\tilde{Z}\longrightarrow \tilde{Z}/\Z_2=Z$. The normal form \Cref{generalnormalform} implies that the manifold with boundary obtained by removing $(-\varepsilon,\,\varepsilon)\times_{\Z_2}\tilde{Z}$ and gluing $(-\varepsilon,\,0]\times \tilde{Z}$ back in has a natural scattering symplectic structure.
\end{proof}
\end{proof}
\Cref{orientableisenough} together with \Cref{boundaryscattering} imply that the first item in \Cref{mainthm} implies the second. We will prove the first item in \Cref{mainthm} in the remaining subsections. In what follows we will always assume orientability of $X$, unless otherwise specified.

\subsection{An alternate normal form theorem}

The following alternate normal form result is going to be useful in the proof of the main theorem. It is stated for orientable scattering symplectic manifolds for simplicity, as that is the only case that we will need. Nonetheless, it holds, with the obvious modifications, for arbitrary scattering symplectic manifolds (compare with \Cref{generalnormalform}).
\begin{prop}\label{normalform2}
Let $(X,\,Z,\,\omega)$ be an orientable scattering symplectic manifold. There exists $\varepsilon>0$, and a symplectic manifold $Y$ with boundary $\partial Y = Z\sqcup Z$, such that the symplectic form extends to a symplectic form $\widehat{\omega}$ on $\widehat{Y}=Y\sqcup [\frac{1}{2\varepsilon^2},\,+\infty)\times Z$. The restriction of $\widehat{\omega}$ to 
$[\frac{1}{2\varepsilon^2},\,+\infty)\times Z$ is 
\begin{equation}\label{normalform2eq}
    d(s\alpha)-\frac{1}{2}\frac{ds}{s}\wedge\beta_1+\beta_2
\end{equation}
and $(\widehat{Y},\,\widehat{\omega})$ is symplectomorphic to $X\mi Z$.
\end{prop}
\begin{proof}
This just follows from \Cref{generalnormalform}. Consider the change of coordinates $s=\frac{1}{2\lambda^2}$. Since $\log(s)=-2\log(x)+\log(2)$, $\frac{dx}{x}=d\log(x)=-\frac{1}{2}d\log(s)$, then the form $\omega$ becomes of the desired form
\begin{equation}
\omega=d(s\alpha)-\frac{1}{2}d\log(s)\wedge\beta_1+\beta_2=d(s\alpha)-\frac{1}{2}\frac{ds}{s}\wedge\beta_1+\beta_2
\end{equation}
\end{proof}
In particular if $\beta_1=\beta_2=0$ that's exactly the normal form at the boundary of a strong filling of $(Z,\,\ker\alpha)$. To be more precise:
\begin{coro}
Write a connected component $\widehat{Y}$ of $X\setminus Z$ as in \Cref{normalform2}; write a neighbourhood of infinity as $[C,\,+\infty)\times Z$ . If $\beta_1=\beta_2=0$, for all $C'>C$, the manifold $W:=\widehat{Y}\cap\{s\leq C'\}$ is a strong filling of $(Z,\,\ker\alpha)$.
\end{coro}

\subsection{Scattering $\Rightarrow$ weak}

The first step to relate scattering to weak filling is to get closer to stable Hamiltonian fillings, by getting rid of the form $\beta_1$.
\begin{prop}\label{deform1}
The symplectic form $\omega$ on $[C,\,+\infty)\times Z$ in \Cref{normalform2eq} can be modified to a symplectic form $\tilde{\omega}$ such that
\begin{itemize}
\item $\omega=\tilde{\omega}$ on $[C,\,C']\times Z$
\item $\tilde{\omega}=d(s\alpha)+\beta_2$ on $[C'',\,+\infty)\times Z$
\end{itemize}
\end{prop}
\begin{proof}
Pick a smooth function $f=f(s):[C,\,+\infty)\times Z\longrightarrow \R$ such that 
\begin{itemize}
\item $f(s)=\frac{1}{s}$ on $[C,\,C']$
\item $f(s)=0$ for $s\geq C''>C'$
\item $f'(s)\leq 0$
\end{itemize}
and define
\[
\tilde{\omega}=d(s\alpha)-\frac{1}{2}f(s)ds\wedge\beta_1+\beta_2
\]
$\tilde{\omega}$ remains closed and nondegenerate, at least if one chooses $C'$ big enough.
\end{proof}
Note that $(\alpha,\,\beta_2)$ is not a stable Hamiltonian structure, as we imposed no condition on $\alpha\wedge\beta^{n-1}$.\\
Note the following consequence.
\begin{coro}
If $(Z,\,\ker\alpha)$ admits a scattering filling, then it admits a filling such that the symplectic form in a collar neighbourhood is of the form $d(s\alpha)+\beta$.
\end{coro}
\begin{proof}
Just realize $Z$ as $\{c\}\times Z$, with $c\geq C''$, in the normal form of \Cref{deform1}. Since we are not changing the symplectic form in a neighbourhood of $\{C\}\times Z$, it extends to the rest of the filling.
\end{proof}
One can reformulate the previous corollary as a statement about scattering symplectic manifold.
\begin{coro}
If a pair $(X,\,Z)$ supports a scattering symplectic structure with cohomology decomposition $(a,\,[\beta_1],\,[\beta_2])$, then it supports one with cohomology decomposition $(a,\,0,\,[\beta_2])$.
\end{coro}
We now apply \Cref{deform1} to show that scattering fillability implies weak fillability.
\begin{coro}
Let $(X,\,Z,\,\omega)$ be an orientable scattering filling of $(Z,\,\ker\alpha)$, with $Z$ connected. Let $U$ be a tubular neighbourhood of $Z$, and $W$ be a component of $X\mi U$. Then there exists a symplectic form $\tilde{\omega}$ so that $(W,\,\tilde{\omega})$ is a weak filling of $(Z,\,\ker\alpha)$.
\end{coro}
\begin{proof}
By \Cref{deform1}, we can attach a cylindrical end to $W$ and get a manifold $\widehat{W}$ with a symplectic form $\tilde{\omega}$, such that $\tilde{\omega}=d(s\alpha)+\beta$ in a neighbourhood $V$ of infinity. Define an almost complex structure $J$ on $TV$ such that
\begin{itemize}
\item $J$ turns $\ker\alpha$ into a complex vector bundle, and $J\partial_s=R_\alpha$, the Reeb vector field of $\alpha$.
\item $d\alpha$ tames $J$ on $\ker\alpha$
\end{itemize}
In particular $J$ tames $d(s\alpha)$ on $V$. If $s$ is big enough, then $J$ also tames $\tilde{\omega}$. This means that there exists a smaller neighbourhood $V'\subset V$ (of the form $(c,\,+\infty)\times Z$) on which $J$ tames $\omega$ and satisfies the properties above. By contractibility of the space of tame complex structures we can extend $J|_V'$ to a tame complex structure on the whole $\widehat{W}$. Choosing $c'>c$, and identifying $Z$ with $\{c'\}\times Z$, we see that \Cref{almostcomplex} implies $\widehat{W}\cap\{s\leq c'\}$ is a weak filling of $(Z,\,\ker\alpha)$.
\end{proof}

\subsection{Weak $\Rightarrow$ scattering}

This is the easier implication. It follows from the following well-known consequence of Moser's argument:
\begin{prop}\label{generalnormalform2}
Let $(W,\,\omega)$ be a symplectic manifold whose boundary $Z$ has a positive cooriented contact structure $\xi=\ker\alpha$. Assume that $\omega|_{\xi}$ is symplectic. Denoting $\omega_Z:=\omega|_{TZ}$, there is a collar neighbourhood $V\cong(-\varepsilon,\,0]\times Z$ such that $\omega\cong d(s\alpha)+\omega_Z$.
\end{prop}
In particular we can assume that the symplectic form on a collar neighbourhood of the boundary in a weak filling can be written as $d(s\alpha)+\omega_Z$. As a consequence one has
\begin{prop}\label{weakscat}
If $(W,\,\omega)$ is a weak filling of $(Z,\,\ker\alpha)$, then there is a symplectic form $\tilde{\omega}$ such that $(W,\,\tilde{\omega})$ is a symplectic manifold with scattering boundary $(Z,\,\ker\alpha)$.
\end{prop}
\begin{proof}
Assume that $\omega=d(s\alpha)+\omega_Z$ in a neighbourhood of the boundary. Using \Cref{generalnormalform2}, attach a copy of $(-\varepsilon,\,+\infty)\times Z$ along a collar neighbourhood of $\partial W$, with symplectic form $d(s\alpha)+\omega_Z$. Denote the resulting symplectic manifold $(\widehat{W},\,\widehat{\omega})$.
The inverse of the change of coordinates in \Cref{deform1} turns $(\widehat{W},\,\widehat{\omega})$ into the open part of a symplectic manifold with scattering boundary, with singular locus $(Z,\,\ker\alpha)$, with $\beta_1=0$, and $\beta_2=\omega_Z$.
\end{proof}
The main theorem follows as a consequence of \Cref{weakscat} and \Cref{boundaryscattering}.

\section{Implications for scattering symplectic manifolds}\label{4}

Consider $q:\tilde{Z}\longrightarrow Z$ the double cover map. If $Z$ is endowed with an overtwisted contact structure (in the sense of \cite{Borman2015}), then the pullback contact structure on $\tilde{Z}$ is also overtwisted. More generally, if $Z$ has a \textit{bordered Legendrian open book (bLob)}  (\cite{weakfillings}), also $\tilde{Z}$ does. Since a bLob is an obstruction to weak fillability, we obtain the following result.
\begin{thm}
A connected contact manifold containing a bLob cannot appear as the singular locus of a (not necessarily orientable) scattering symplectic manifold. In particular, the singular locus of a scattering symplectic manifold, if connected, is tight (i.e. not overtwisted).
\end{thm}
Moreover, since every weak filling can be deformed to a stable Hamiltonian filling (\cite{weakfillings}), one obtains a well behaved theory of possibly punctured holomorphic curves with values in a scattering symplectic manifold, the punctures being asymptotic to closed Reeb orbits in the singular locus (with respect to the deformed structure).

\printbibliography
\end{document}